\documentclass[12pt]{amsart}
\usepackage{amsmath}

\newtheorem{theorem}{Theorem}  
\newtheorem{lemma}[theorem]{Lemma}
\newtheorem{proposition}[theorem]{Proposition}

\newtheorem{remar}[theorem]{Remark}
\renewenvironment{proof}{Proof:\ \ \ }{\QED}

\newcommand{\QED}{\qed}
\newcommand{\bfind}[1]{\index{#1}{\sl #1}}
\newcommand{\n}{\par\noindent}

\newcommand{\sn}{\par\smallskip\noindent}
\newcommand{\mn}{\par\medskip\noindent}

\newcommand{\pars}{\par\smallskip}
\newcommand{\parm}{\par\medskip}

\newcommand{\subsetuneq}{\mathrel{\raisebox{.8ex}{\footnotesize%
$\displaystyle\mathop{\subset}_{\not=}$}}}

\newcommand{\R}{\mathbb R}
\newcommand{\Q}{\mathbb Q}
\newcommand{\N}{\mathbb N}
\newcommand{\Z}{\mathbb Z}
\newcommand{\cB}{\mathcal B}
\newcommand{\cN}{\mathcal N}
\newcommand{\cal}{\mathcal}

\begin{document}
\title[Fixed Point Theorems]{A common generalization of metric,
ultrametric and topological fixed point theorems\\[.3cm]
\normalsize\it --- alternative version --- }

\author{Katarzyna Kuhlmann and Franz-Viktor Kuhlmann}
\address{Department of Mathematics \& Statistics, University of
Saskatchewan, 106 Wiggins Road, Saskatoon, SK, S7N 5E6, Canada}
\email{fvk@math.usask.ca}
\address{Institute of Mathematics, Silesian University, Bankowa 14,
40-007 Katowice, Poland}
\email{kmk@math.us.edu.pl}
\date{March 22, 2013}
\thanks{The authors wish to thank the referee for his extremely careful
reading and several very helpful and essential corrections and
suggestions. They also thank Ed Tymchatyn, Murray Marshall and John
Martin for inspiring discussions.\\
\indent
The research of the first author was partially supported
by a sabbatical grant from the Silesian University at Katowice, Poland.
The research of the second author was partially supported
by a Canadian NSERC grant.}

\begin{abstract}
We present a general fixed point theorem which can be seen as the
quintessence of the principles of proof for Banach's Fixed Point
Theorem, ultrametric and certain topological fixed point theorems. It
works in a minimal setting, not involving any metrics. We demonstrate
its applications to the metric, ultrametric and topological cases, and
to ordered abelian groups and fields.
\end{abstract}
\maketitle


%
%
\section{Introduction}
What is the common denominator of Banach's Fixed Point Theorem and its
ultrametric and topological analogues as developed in
\cite{[P],[PR1],[PR2],[KU4]} and in \cite{[SWJ]}? Is there a general
principle of proof that works for all of these worlds, the (ordinary)
metric, ultrametric and topological, and beyond? In this paper, we give
an answer to these questions. We draw our inspiration from the notions
of ``ball'' and ``spherical completeness'' that are used in the
ultrametric world.

S.~Priess' paper \cite{[P]} in which she first proved a fixed point
theorem for ultrametric spaces initiated an interesting development that
led to a better understanding of important theorems in valuation theory
and to new results (see, e.g., [KU4]). This was achieved by extracting
the underlying principle of the proof of Hensel's Lemma through
abstracting from the algebraic operations and only considering the
ultrametric induced by the valuation. In this paper we push this
development one step further by extracting the underlying principle of
various fixed point theorems. In this way, a general framework is set
up that helps understand these theorems in a more conceptual manner and
to transfer ideas from one world to the other by analogies (as we will
demonstrate, for instance, for the topological fixed point theorem we
consider).

The general framework also helps to make the use of fixed point theorems
available to situations that are difficult or even impossible to subsume
under the above mentioned settings. While investigating spaces of real
places, we found that in certain algebraic entities, it may be much
easier and more natural to define ``balls'' than to define the
``distance'' between two elements. For example, if we are dealing with
quotient topologies, like in the case of spaces of real places where the
topology is induced by the Harrison topology of spaces of orderings,
balls come up naturally as images of certain open or closed sets in the
inducing topology. Therefore, we will work with \bfind{ball spaces}
$(X,\cB)$, that is, sets $X$ with a nonempty set $\cB$ of distinguished
nonempty subsets of $X$, which we call \bfind{balls}. We require no
further structure on these spaces. We do not even need a topology
generated in some way by the balls. But let us mention that the way we
formulate our theorems, the balls should be considered closed, rather
than open, in such a topology, because singletons appear and are
important. One can reformulate everything in an ``open ball'' approach,
but this makes the exposition less elegant.

We found the idea of centering the attention on balls, rather than
metrics, in the paper \cite{tt}. But there, ball spaces carry much more
structure and the conditions for a fixed point theorem are
unnecessarily restrictive.

We will now state our most general fixed point theorem for ball spaces.
We need two notions. A \bfind{nest of balls} in $(X,\cB)$ is a
nonempty collection of balls in $\cB$ that is totally ordered by
inclusion. If $f:X\rightarrow X$ is a function, then a subset
$B\subseteq X$ will be called \bfind{$f$-contracting} if it is either a
singleton containing a fixed point or satisfies $f(B)\subsetuneq B$.

\begin{theorem}                             \label{NFPT1}
Take a function $f$ on a ball space $(X,\cB)$ which satisfies the following
conditions:
\sn
{\bf (C1)} \ there is at least one $f$-contracting ball,\n
{\bf (C2)} \ for every $f$-contracting ball $B\in\cB$, the image $f(B)$
contains an $f$-contracting ball,\n
{\bf (C3)} \ the intersection of every nest of $f$-contracting balls
contains an $f$-contracting ball.
\sn
Then $f$ admits a fixed point.
\end{theorem}

We can obtain uniqueness of the fixed point by strengthening the
hypothesis:

\begin{theorem}                             \label{NFPT2}
Take a function $f$ on a ball space $(X,\cB)$ which satisfies the
following conditions:
\sn
{\bf (CU1)} \ $X$ is an $f$-contracting ball,\n
{\bf (CU2)} \ for every $f$-contracting ball $B\in\cB$, the image $f(B)$
is again an $f$-con\-tracting ball,\n
{\bf (CU3)} \ the intersection of every nest of $f$-con\-tracting balls
is again an $f$-con\-tracting ball.
\sn
Then $f$ has a unique fixed point.
\end{theorem}
These theorems will be proved in Section~\ref{sectproof}.

\pars
In \cite{[SWJ]} the authors show that every ``$J$-contraction'' on a
connected compact Hausdorff space $X$ has a unique fixed point. Using
the inspiration from our general framework, we obtain the following
strong generalization; note that we do not require the space $X$ to be
Hausdorff.

\begin{theorem}                             \label{topn}
Take a compact space $X$ and a closed function $f:X\rightarrow X$.
If every nonempty closed set $B$ in $X$ with $f(B)\subseteq B$ contains
a closed $f$-contracting subset, then $f$ has a fixed point in $X$. If
every nonempty closed set $B$ in $X$ with $f(B)\subseteq B$ is
$f$-contracting, then $f$ has a unique fixed point in $X$.
\end{theorem}

This theorem and the theorem of \cite{[SWJ]} are corollaries to
Theorems~\ref{NFPT1} and~\ref{NFPT2}. We will show this in
Section~\ref{secttop}, where we will also present another version of
Theorem~\ref{topn} that is directly related to Theorem~\ref{GFPT2}
below.

\parm
For most applications of these theorems, it is an advantage to have a
handy criterion for the existence of the $f$-contracting balls. From
classical fixed point theorems we know the assumption that the function
$f$ be strictly contracting. But we have learnt from the ultrametric
case that if one does not insist on uniqueness, one can relax the
conditions: a function does not need to be strictly contracting on the
whole space, but only on the orbits of its elements (and simply
contracting otherwise). While this relaxation makes the formulation of
the conditions a bit longer, it should be noted that it is important for
many applications in which the function under consideration fails for
natural reasons to be strictly contracting on the whole space. The
following way to present a fixed point theorem may seem unusual, but it
turns out to be very close to several applications as it encodes (a
weaker form of) the property ``strictly contracting on orbits''.

Consider a function $f:X\rightarrow X$. We will write $fx$ for $f(x)$ and
$f^i x$ for the image of $x$ under the $i$-th iteration of $f$, that is,
$f^0x=x$, $f^1x=f(x)$ and $f^{i+1}x=f(f^ix)$. The function $f$ will be
called \bfind{strongly contracting on orbits} if there is a function
\[
X\ni x\>\mapsto\>B_x\in {\cal B}
\]
such that for all $x\in X$, the following conditions hold:
\sn
{\bf (SC1)} \ $x\in B_x\,$, \n
{\bf (SC2)} \ $B_{fx}\subseteq B_x\,$, and if $x\ne fx$, then
$B_{f^ix}\subsetuneq B_x$ for some $i\geq 1$.
\sn
Note that (SC1) and (SC2) imply that $f^i x\in B_x$ for all $i\geq 0$.

We will say that a nest of balls $\cN$ is an \bfind{$f$-nest} if $\cN=
\{B_x\mid x\in S\}$ for some set $S\subseteq X$ that is closed under
$f$. Now we can state our third main theorem:

\begin{theorem}                             \label{GFPT2}
Take a function $f$ on a ball space which is strongly contracting on
orbits. If for every $f$-nest $\cN$ in this ball space there is some
$z\in \bigcap\cN$ such that $B_z\subseteq\bigcap\cN$, then $f$ has a
fixed point.
\end{theorem}

Theorem~\ref{GFPT2} does not deal with the question of uniqueness of
fixed points; this question is answered in the particular applications
by additional arguments that are often very easy.

The condition about the intersection of an $f$-nest is not needed for
Banach's Fixed Point Theorem and may therefore appear alien to readers
who are not familiar with the ultrametric case. But there, as in the
case of non-archimedean ordered groups and fields, one has to deal with
jumps that one could intuitively think of as being a ``non-archimedean''
or ``non-standard'' phenomenon. The obstruction is that the intersection
of an infinite nest of balls we have constructed may contain more than
one element, at which point we have to iterate the construction. The
mentioned condition makes this work.

\pars
The condition on the intersection of $f$-nests implies that in
particular, they are not empty. This reminds of a similar property of
ultrametric spaces, and we take over the corresponding notion. The ball
space $(X,\cB)$ will be called \bfind{spherically complete} if every
nonempty nest of nonempty balls has a nonempty intersection.

To illustrate the flexibility of the concepts we have introduced and the
above explained idea of making fixed point theorems available to totally
new settings, we state the following easy but useful result:

\begin{proposition}
Take two ball spaces $(X_1,\cB_1)$ and $(X_2,\cB_2)$ and a function
$f:X_1\rightarrow X_2\,$. Suppose that the preimage of every ball in
$\cB_2$ is a ball in $\cB_1\,$. If $\cN$ is a nest of balls in
$(X_2,\cB_2)$, then the preimages of the balls in $\cN$ form a nest of
balls in $(X_1,\cB_1)$. If $(X_1,\cB_1)$ is spherically complete, then
so is $(X_2,\cB_2)$.
\end{proposition}

\pars
In several applications, and in particular in the ultrametric setting,
the function under consideration has in a natural way stronger
properties than we have used so far. What we have asked for one element
in the intersection of an $f$-nest is often satisfied by every element
in the intersection. Therefore, it seems convenient to introduce a
notion which reflects this property and in this way to separate it from
the condition that the intersections is non-empty.
The function $f$ will be called \bfind{self-contractive} if in
addition to (SC1) and (SC2), it satisfies:
\sn
{\bf (SC3)} \ if $\cN$ is an $f$-nest and if $z\in\bigcap \cN$, then
$B_z\subseteq\bigcap \cN$.
\sn
Self-contractive functions will appear in the hypothesis of
Theorem~\ref{AT1}, in Theorem~\ref{top3}, in Theorem~\ref{SUFPT}, and in
the proof of Banch's Fixed Point Theorem. The following fixed point
theorem is an easy corollary to Theorem~\ref{GFPT2}:

\begin{theorem}                             \label{GFPT1}
Every self-contractive function on a spherically complete ball space has
a fixed point.
\end{theorem}

\parm
For the proof of Theorems~\ref{NFPT1}, \ref{NFPT2} and~\ref{GFPT2}, see
Section~\ref{sectproof}. In Section~\ref{sectgat}, we state two general
attractor theorems. Section~\ref{secttop} is devoted to topological
fixed point theorems. In Section~\ref{sectult}, we show how to derive
ultrametric fixed point theorems, and in Section~\ref{sectattr}, we
discuss ultrametric attractor theorems. In
Section~\ref{sectcbs} we then discuss valued fields that are complete by
stages, a notion introduced by P.~Ribenboim in \cite{[R]}. We use
Theorem~\ref{GFPT2} for a quick proof of a fixed point theorem that
works in such fields (Theorem~\ref{FPTcbs}). This theorem can be used to
show that such fields are henselian. Its application to the proof of
Hensel's Lemma provides an example for a case where one does not have in
any natural way a function that is strictly contracting on all of the
space. Note also that the particularly weak form that we have chosen for
(SC2) comes in very handy for the formulation of Theorem~\ref{FPTcbs}.

\pars
In Section~\ref{sectBFT} we discuss how to derive Banach's Fixed Point
Theorem. Our aim is not to provide a new proof of this theorem; in
contrast to our other applications, the existing proofs in this case are
much shorter. Our aim here is to show how to convert the problem from
metric to ball space and to pave the way for one of our fixed point
theorems (Theorem~\ref{OAG}) for ordered abelian groups and fields.
Associated with them are two natural ball spaces:
\sn
$\bullet$ \ the order ball space, where the balls are closed bounded
intervals, and
\n
$\bullet$ \ the ultrametric ball space, where the balls are the
ultrametric balls derived from the natural valuation.
\sn
We discuss these ball spaces and the corresponding fixed point theorems
in Section~\ref{sectna}. The flexibility of our notion of ball space is
demonstrated in the concept of hybrid ball spaces, in which we use order
balls and ultrametric balls simultaneously. One of such hybrid ball
spaces is used for a simple characterization of those ordered fields
which are power series fields with residue field $\R$
(Theorem~\ref{charpsfR}).

%
%
\section{Proof of the fixed point theorems for ball spaces}
\label{sectproof}
\noindent
{\bf Proof of Theorem~\ref{NFPT1}}:\n
The set of all nests consisting of $f$-contracting balls is partially
ordered by inclusion. There is at least one such nest since by condition
(C1), there is at least one $f$-contracting ball. Further, the union
over an ascending chain of nests consisting of $f$-contracting balls is
again such a nest (observe that the cardinality of this union is bounded
by the cardinality of the power set of $X$, so the union is a set).
Hence by Zorn's Lemma, there is a maximal nest $\cN$ consisting of
$f$-contracting balls. By condition (C3), $\bigcap\cN$ contains an
$f$-contracting ball $B$. Suppose this ball is not a singleton. But
then by condition (C2), $f(B)$ contains an $f$-contracting ball $B'$.
Since $B'\subseteq f(B)\subsetuneq B$, $\cN\cup\{B'\}$ is then a nest
that properly contains $\cN$, which contradicts the maximality. We find
that $B$ must be a singleton consisting of a fixed point.       \qed

\mn
{\bf Proof of Theorem~\ref{NFPT2}}:\n
Using conditions (CU1), (CU2), (CU3) and transfinite induction, we build
a nest $\cN$ consisting of $f$-contracting balls as follows. We set
$\cN_0:=\{X\}$. Having constructed $\cN_\nu$ for some ordinal $\nu$ with
smallest $f$-contracting ball $B_\nu\in\cN_\nu\,$, we set $B_{\nu+1}:=
f(B_\nu)$ and $\cN_{\nu+1}:=\cN_{\nu} \cup \{B_{\nu+1}\}$. If $\lambda$
is a limit ordinal and we have constructed $\cN_\nu$ for all
$\nu<\lambda$, we observe that the union over all $\cN_\nu$ is a nest
$\cN'_\lambda\,$. We set $B_\lambda:=\bigcap\cN'_\lambda$ and
$\cN_\lambda:=\cN'_\lambda \cup \{B_\lambda\}$.

If $B_\nu$ is not a singleton, then $B_{\nu+1}\subsetuneq B_\nu\,$.
Hence there must be an ordinal $\nu$ of cardinality at most that of $X$
such that $B_{\nu+1}= B_\nu\,$. But this only happens if $B_\nu$ is a
singleton consisting of a fixed point $x$. If $x\ne y\in X$, then
$y\notin B_\nu$ which means that there is some $\mu<\nu$ such that
$y\in B_\mu$, but $y\notin B_{\mu+1}=f(B_\mu)$. This shows that $y$
cannot be a fixed point of $f$. Therefore, $x$ is the unique
fixed point of $f$.                    \qed

\parm
Theorem~\ref{GFPT2} can be derived from Theorem~\ref{NFPT1}. But as it
takes essentially the same effort, we will give a proof along the lines
of the proof of Theorem~\ref{NFPT1}.
\sn
{\bf Proof of Theorem~\ref{GFPT2}}:\n
Take a function $f$ on the ball space $(X,\cB)$ which is contractive on
orbits. For every $x\in X$, the set $\{B_{f^ix}\mid i\geq 0\}$ is an
$f$-nest. The set of all $f$-nests is partially ordered by inclusion.
The union over an ascending chain of $f$-nests is again an
$f$-nest. Hence by Zorn's Lemma, there is a maximal $f$-nest $\cN$. By
the assumption of Theorem~\ref{GFPT2}, there is some $z\in \bigcap\cN$
such that $B_z\subseteq\bigcap\cN$. We wish to show that $z$ is a fixed
point of $f$. If we would have that $z\ne fz$, then by (SC2),
$B_{f^iz}\subsetuneq B_z\subseteq\bigcap\cN$ for some $i\geq 1$, and the
$f$-nest $\cN\cup\{B_{f^kz}\mid k\in\N\}$ would properly contain $\cN$.
But this would contradict the maximality of $\cN$. Hence, $z$ is a fixed
point of $f$. \qed

%
%
\section{General attractor theorems}        \label{sectgat}
Let us derive from Theorem~\ref{GFPT1} an attractor theorem which is
modeled after the ultrametric attractor theorem in \cite{[KU3]}. We
consider two ball spaces $(X,\cB)$ and $(X',\cB')$ and a function
$\varphi: X\rightarrow X'$. Take an element $z'\in X'$. If there is a
function $f:X\rightarrow X$ which is strongly contracting on orbits,
and a function
\[
X\ni x\>\mapsto\>B'_x\in\cB'
\]
such that for all $x\in X$, the following conditions hold:
\sn
{\bf (AT1)} \ $z'\in B'_x$ and $\varphi(B_x)\subseteq B'_x$, \n
{\bf (AT2)} \ if $\varphi(x)\ne z'$, then $B'_{f^i x}\subsetuneq
B'_x$ for some $i\in\N$,
\sn
then $z'$ will be called a \bfind{weak $f$-attractor for $\varphi$}. If
in addition $f$ is self-contractive, then $z'$ will be called an
\bfind{attractor for $\varphi$}.

\begin{theorem} {\bf (Attractor Theorem 1)} \label{AT1} \n
Take a function $\varphi: X\rightarrow X'$ and an attractor $z'\in X'$
for $\varphi$. If $(X,\cB)$ is spherically complete, then $z'\in\varphi
(X)$.
\end{theorem}
Indeed, by Theorem~\ref{GFPT1}, $f$ has a fixed point $z$. But by
condition (AT2), $fz=z$ implies that $\varphi(z)= z'$. The following
version of the Attractor Theorem follows in a similar way from
Theorem~\ref{GFPT2}:

\begin{theorem} {\bf (Attractor Theorem 2)} \label{AT2} \n
Take a function $\varphi: X\rightarrow X'$ and a weak $f$-attractor
$z'\in X'$ for $\varphi$. If for every $f$-nest $\cN$ in $(X,\cB)$
there is some $z\in \bigcap\cN$ such that $B_z\subseteq\bigcap\cN$, then
$z'\in\varphi (X)$.
\end{theorem}

%
%
\section{Fixed point theorems for topological spaces}
\label{secttop}
In this section, we consider compact topological spaces $X$ with
functions $f: X \rightarrow X$. We note:

\begin{lemma}                               \label{cHsc}
Every compact space $X$ together with any family
of nonempty closed subsets is a spherically complete ball space.
\end{lemma}
\begin{proof}
\cite[Proposition 2, p.\ 57]{ap} states that every ``centered system''
of nonempty closed subsets of a compact space $X$ has a nonempty
intersection. Here, a ``centered system'' means a family of subsets that
is linearly ordered under inclusion. This is exactly what we call a
``nest''. Therefore, the cited proposition proves our lemma.
\end{proof}

\pars
In view of this lemma, we will take $\cB$ to be the set of all nonempty
closed subsets of $X$. We show how to deduce Theorem~\ref{topn} from
Theorem~\ref{NFPT1}.

\pars
{\bf Proof of Theorem~\ref{topn}}:
Take a compact space $X$ and a closed function $f:X\rightarrow X$.
Assume first that every closed set $B$ in $X$ with $f(B)\subseteq B$
contains a closed $f$-contracting subset. Since $X$ is closed, this
implies that condition (C1) is satisfied. If $B$ is an $f$-contracting
closed set in $X$, then $f(B)$ is closed since $f$ is a closed function.
Also, we have that $f(B)\subseteq B$, which yields that $f(f(B))
\subseteq f(B)$. Hence by assumption, $f(B)$ contains an $f$-contracting
closed set, so condition (C2) is satisfied. The intersection $\bigcap
{\cal N}$ of a nest ${\cal N}=\{B_i\mid i\in I\}$ of $f$-contracting
closed sets $B_i$ is closed; since $f(B_i)\subseteq B_i$ for all
$i\in I$, we also have that $f(\bigcap {\cal N}) \subseteq \bigcap
{\cal N}$. Hence by assumption, $\bigcap {\cal N}$ contains an
$f$-contracting closed set. So condition (C3) is satisfied, and
Theorem~\ref{NFPT1} shows that $f$ has a fixed point.

\pars
Now assume that every closed set $B$ in $X$ with $f(B)\subseteq B$ is
$f$-contracting. Then $X$ is $f$-contracting, and for every
$f$-contracting closed set $B$ also $f(B)$ is closed with $f(f(B))
\subseteq f(B)$ and hence $f$-contracting. Further, the intersection
$\bigcap {\cal N}$ of a nest of $f$-contracting closed sets is
closed with $f(\bigcap {\cal N}) \subseteq \bigcap {\cal N}$ and hence
$f$-contracting. Therefore, Theorem~\ref{NFPT2} shows that $f$ has a
unique fixed point.                                          \qed

\pars
We will now show how to deduce a fixed point theorem of \cite{[SWJ]}
from Theorem~\ref{topn}. For this theorem, we assume that $X$ is
compact, Hausdorff and connected. An open cover $\mathcal U$ of $X$ is
said to be {\em $J$-contractive for $f$} if for every $U \in \mathcal U$
there is $U' \in \mathcal U$ such that $f(\mbox{\rm cl}\, U) \subseteq
U'$, where $\mbox{\rm cl}\, U$ denotes the closure of $U$. The function
$f: X \rightarrow X$ is called {\em $J$-contraction} if every open cover
$\mathcal U$ has a finite $J$-contractive open refinement $\mathcal V $
for $f$. For every non-connected compact Hausdorff space $X$ there is a
$J$-contraction of $X$ which has no fixed points (cf.~\cite[Proposition
3, p.\ 553]{[SWJ]}); therefore, the approach using $J$-contractions only
works for connected compact Hausdorff spaces. We cite two important
facts about a $J$-contraction $f$ on a connected compact Hausdorff space
$X$:
\begin{enumerate}
\item[(J1)] If $B$ is a closed subset of $X$ with $f(B)\subseteq B$,
then the restriction of $f$ to $B$ is also a $J$-contraction
(\cite[Proposition 1, p.\ 552]{[SWJ]});
\item[(J2)] If $f$ is onto, then $|X|$ = 1 (\cite[Proposition 4, p.\
554]{[SWJ]}).
\end{enumerate}

The following is Theorem 4 of \cite{[SWJ]}:

\begin{theorem}                             \label{SWJ}
Take a connected compact Hausdorff space $X$ and a continuous
$J$-contraction $f:X\rightarrow X$. Then $f$ has a unique fixed point.
\end{theorem}
\begin{proof}
We claim that by (J1) and (J2), every closed subset $B$ of $X$ with
$f(B) \subseteq B$ is $f$-contracting. Indeed, since by (J1), $f|_B$
is a $J$-contraction on $B$, (J2) shows that either $f|_B$ is not onto,
or $B$ is a singleton $\{x\}$ and since $fx\in f(B)\subseteq B$, we have
that $fx=x$. Now Theorem~\ref{SWJ} follows from Theorem~\ref{topn}.
\end{proof}

\parm
The next theorem shows how to apply Theorem~\ref{GFPT2} to toplogical
spaces.


\begin{theorem}                             \label{top3}
Take a compact space $X$ and a closed function $f:X\rightarrow X$.
Assume that for every $x\in X$ with $fx\ne x$ there is a closed subset
$B$ of $X$ such that $x\in B$ and $x\notin f(B)\subseteq B$. Then $f$
has a fixed point in $B$. Moreover, $f$ is self-contractive, and for
every $x\in X$ with $fx\ne x$ there is a smallest closed subset $B$ of
$X$ such that $x\in B$ and $x\notin f(B)\subseteq B$.
\end{theorem}
\begin{proof}
For every $x\in X$ we consider the following family of
balls:
\[
\mathfrak B_x := \{B \mid B \text{ closed subset of $X$, }
x\in B \text{ and } f(B)\subseteq B\}.
\]
Note that $\mathfrak B_x$ is nonempty because it contains $X$.
We define
\[
B_x := \bigcap\mathfrak B_x\;.
\]
We see that $x \in B_x$ and that $f(B_x)\subseteq B_x$. Further, $B_x$
is closed, being the intersection of closed sets. This shows that $B_x$
is the smallest member of $\mathfrak B_x\,$.

For every $B\in \mathfrak B_x$ we have that $fx\in B$ and therefore,
$B\in \mathfrak B_{fx}\,$. Hence we find that $B_{fx} \subseteq B_x$.

Assume that $fx\ne x$. Then by hypothesis, there is a closed set $B$ in
$X$ such that $x\in B$ and $x\notin f(B)\subseteq B$. Since $f$ is a
closed function, $f(B)$ is closed. Moreover, $f(f(B)) \subseteq f(B)$
and $fx \in f(B)$, so $f(B) \in \mathfrak B_{fx}$. Since $x\notin f(B)$,
we conclude that $x\notin B_{fx}$, whence $B_{fx} \subsetuneq B_x\,$. We
have now proved that $f$ is strongly contracting on orbits. Further,
$B\in \mathfrak B_x\,$, whence $B_x\subset B$, $f(B_x)\subset f(B)$ and
therefore, $x\notin f(B_x)$. This shows that $B_x$ is the smallest of
all closed sets $B$ in $X$ for which $x\in B$ and $x\notin f(B)\subseteq
B$.

Take an $f$-nest $\cN$. Lemma~\ref{cHsc} shows that $\bigcap \cN$ is
nonempty. Take any $z \in \bigcap \cN$. Choose an arbitrary $B_x \in
\cN$. Then $z \in B_x$ and thus, $B_x\in \mathfrak B_z\,$. So we have
that $B_z \subseteq B_x\,$. Therefore, $B_z \subseteq\bigcap
\cN$. We have proved that $f$ is self-contractive.

Theorem~\ref{top3} now follows from Theorem~\ref{GFPT2}.
\end{proof}

%
%
\section{Ultrametric fixed point theorems}  \label{sectult}
Let $(X,d)$ be an ultrametric space. That is, $d$ is a function from
$X\times X$ to a partially ordered set $\Gamma$ with smallest element
$0$, satisfying that for all $x,y,z\in X$ and all $\gamma\in\Gamma$,
\sn
{\bf (U1)} \ $d(x,y)=0$ \ if and only if \ $x=y$,\n
{\bf (U2)} \ if $d(x,y)\leq\gamma$ and $d(y,z)\leq\gamma$, then
$d(x,z)\leq\gamma$,\n
{\bf (U3)} \ $d(x,y)=d(y,x)$ \ \ \ (symmetry).
\sn
(U2) is the ultrametric triangle law; if $\Gamma$ is totally ordered, it
can be replaced by
\sn
{\bf (UT)} \ $d(x,z)\leq\max\{d(x,y),d(y,z)\}$.
\sn
We obtain the \bfind{ultrametric  ball space} $(X,\cB_u)$ from $(X,d)$
by taking $\cB_u$ to be the set of all
\[
B(x,y)\>:=\>\{z\in X\mid d(x,z)\leq d(x,y)\}\;.
\]

It follows from the ultrametric triangle law that $B(x,y)=B(y,x)$ and
that
\begin{equation}                            \label{umball}
B(t,z)\subseteq B(x,y) \ \ \mbox{ if and only if } \ \
t\in B(x,y) \mbox{ and } d(t,z)\leq d(x,y)\;.
\end{equation}
In particular,
\[
B(t,z)\>\subseteq\>B(x,y)\quad\mbox{ if }\; t,z\in B(x,y)\;.
\]

Two elements $\gamma$ and $\delta$ of $\Gamma$ are \bfind{comparable} if
$\gamma\leq\delta$ or $\gamma\geq\delta$. Hence if $d(x,y)$ and $d(y,z)$
are comparable, then $B(x,y)\subseteq B(y,z)$ or $B(y,z) \subseteq
B(x,y)$. If $d(y,z)<d(x,y)$, then in addition, $x\notin B(y,z)$ and
thus, $B(y,z) \subsetuneq B(x,y)$. We note:
\begin{equation}                            \label{cb}
d(y,z)\,<\,d(x,y)\>\Longrightarrow\>B(y,z)\,\subsetuneq\, B(x,y)\;.
\end{equation}

If $\Gamma$ is totally ordered and $B_1$ and $B_2$
are any two balls with nonempty intersection, then $B_1\subseteq B_2$ or
$B_2\subseteq B_1\,$.

The ultrametric space $(X,d)$ is called \bfind{spherically complete}
if the corresponding ball space is spherically complete.
%
%
The following theorem (with $i=1$ in (\ref{scoo})) appeared in
\cite{[PR2]}:

\begin{theorem} {\bf (Strong Ultrametric Fixed Point Theorem)}
\label{SUFPT} \n
Take a spherically complete ultrametric space $(X,d)$ and a function
$f:X\rightarrow X$. Assume that $f$ satisfies,  for all $x,z\in X$:
\begin{eqnarray}
&& x\ne fx \>\Longrightarrow\>\exists i\geq 1:\>
d(f^ix,f^{i+1}x)<d(x,fx)\>,           \label{scoo}       \\
&& d(z,fx)\leq d(fx,f^2x) \>\Longrightarrow\>
d(z,fz)\leq d(x,fx)\>.                            \label{SUFPTc}
\end{eqnarray}
%
%
Then $f$ has a fixed point.
\end{theorem}
\begin{proof}
Our theorem follows from Theorem~\ref{GFPT1} once we have shown that $f$
is self-contractive on the ball space $(X,\cB_u)$. We define
\begin{equation}                            \label{uBx}
B_x\>:=\>B(x,fx)
\end{equation}
and observe that $x\in B_x\,$. Taking $z=fx$ in (\ref{SUFPTc}), we find
that $d(fx,f^2x)\leq d(x,fx)$. Hence by (\ref{umball}), $B_{fx}=
B(fx,f^2x) \subseteq B(x,fx)=B_x$. By induction on $i$ it follows that
$f^i x\in B_x\,$. By (\ref{scoo}), $d(f^ix,f^{i+1}x)<d(x,fx)$ for some
$i\geq 1$. Then by (\ref{cb}), we have that $B_{f^ix}= B(f^i x,
f^{i+1}x) \subsetuneq B(x,fx)=B_x$. So we have proved that $f$ satisfies
(SC1) and (SC2).

To show that also (SC3) holds, we take an $f$-nest $\cN$ and any $z\in
\bigcap\cN$. We have to show that $B_z \subseteq\bigcap\cN$, that is,
$B_z \subseteq B_x$ for all $B_x\in\cN$. Since $z\in \bigcap\cN
\subseteq B_{fx}=B(fx,f^2x)$, we have that $d(z,fx)\leq d(fx,f^2x)$. By
(\ref{SUFPTc}), this implies that $d(z,fz)\leq d(x,fx)$. Since we know
that $z\in B_x\,$, (\ref{umball}) now shows that $B_z =B(z,fz)\subseteq
B(x,fx)=B_x\,$.
\end{proof}

A function $f:X\rightarrow X$ is called \bfind{contracting} if $d(fx,fy)
\leq d(x,y)$ for all $x,y\in X$. It is shown in \cite{[PR2]} that the
following theorem (in the case of $i=1$ in (\ref{scoo})) follows from
Theorem~\ref{SUFPT}:

\begin{theorem}  {\bf (Ultrametric Fixed Point Theorem)} \label{UFPT}\n
Every contracting function on a spherically complete ultrametric space
which satisfies (\ref{scoo}) has a fixed point.
\end{theorem}

This theorem follows directly from Theorem~\ref{GFPT1} by way of the
following result:

\begin{lemma}                               \label{csco}
Take a contracting function $f$ on an ultrametric space $(X,d)$ and
define the balls $B_x$ as in (\ref{uBx}). Then $f$ satisfies (SC3), and
$f(B_x)\subseteq B_x$ and $B_{fx}\subseteq B_x$ for all $x\in X$. If
$f$ also satisfies (\ref{scoo}), then it is self-contractive.
\end{lemma}
\begin{proof}
We claim that for contracting $f$ we have, as in the topological case,
\begin{equation}                            \label{zinBx}
B_z\>\subseteq\> B_x \quad\mbox{for all } z\in B_x\:.
\end{equation}
Indeed, $z\in B_x$ means that $d(x,z)\leq d(x,fx)$. Since $f$ is
contracting, we then have that $d(fx,fz)\leq d(x,z)\leq d(x,fx)$.
Together with the trivial inequality $d(x,fx)\leq d(x,fx)$ and the
ultrametric triangle law, this yields that $d(x,fz)\leq d(x,fx)$.
Together with $d(z,x)=d(x,z)\leq d(x,fx)$ and the ultrametric triangle
law, this yields that $d(z,fz)\leq d(x,fx)$. Now (\ref{umball}) shows
that $B_z= B(z,fz)\subseteq B(x,fx)=B_x\,$, which proves (\ref{zinBx}).
We also obtain that $fz\in B_x\,$, and as $z\in B_x$ was arbitrary, this
shows that $f(B_x)\subseteq B_x\,$.

Taking $z=fx\in B_x$ in (\ref{zinBx}), we find that $B_{fx}\subseteq
B_x\,$.

If $\cN$ is an $f$-nest and $z\in\bigcap\cN$, then for every $B_x\in
\cN$ we have that $z\in B_x$ and by (\ref{zinBx}), $B_z\subseteq B_x\,$.
This implies that $B_z\subseteq\bigcap\cN$, which proves (SC3).

The last assertion of the lemma is clear.
%
\end{proof}

%

%
%
\section{Ultrametric attractor theorems}    \label{sectattr}
In this section, we present a generalization of the attractor theorem of
\cite{[KU3]} to ultrametric spaces with partially ordered value sets,
and show how to derive it from Theorem~\ref{AT1}.

Take ultrametric spaces $(X,d)$ and $(X',d')$ and a function
$\varphi:\;X \rightarrow X'$. An element $z'\in X'$ is called
\bfind{attractor for $\varphi$} if for every $x\in X$ such that $z'\ne
\varphi x$, there is an element $y\in X$ which satisfies:
\sn
{\bf (UAT1)} \ $d'(\varphi y,z')<d'(\varphi x,z')$,\n
{\bf (UAT2)} \ $\varphi(B(x,y))\subseteq B(\varphi x,z')$,\n
{\bf (UAT3)} \ if $t\in X$ such that $d'(\varphi x,z')<d'(\varphi t,z')$
and $\varphi(B(t,x))\subseteq B(\varphi t,z')$, then $d(t,x)$ and
$d(x,y)$ are comparable.
\sn
Condition (UAT1) says that the approximation $\varphi x$ of $z'$ from
within the image of $\varphi$ can be improved, and condition (UAT2) says
that this can be done in a somewhat continuous way. Condition (UAT3) is
always satisfied when the value set of $(X,d)$ is totally ordered,
which implies that any two balls with nonempty intersection are
comparable by inclusion. For this reason, it does not appear as a
condition in the attractor theorem of \cite{[KU3]}. But if
the value set of $(X,d)$ is not totally ordered, then it can happen that
several ``parallel universes'' exist around a point; (UAT3) then
guarantees that we can keep our approximations to remain in the same
universe.

\begin{theorem}                             \label{MTattr}
Assume that $z'\in X'$ is an attractor for $\varphi:\;X \rightarrow X'$
and that $(X,d)$ is spherically complete. Then $z'\in \varphi(X)$.
\end{theorem}
\begin{proof}
For $x\in X$, we define $B'_x:=B'(\varphi x,z')$. Then we define a
function $f:\;X \rightarrow X$ as follows. If $\varphi x=z'$, we set
$fx=x$. If $\varphi x\ne z'$, then we choose some $y\in X$ that
satisfies (UAT1), (UAT2), (UAT3) and set $fx:=y$. In both cases,
we set $B_x:=B(x,fx)$.

We have that $z'\in B'_x$ by definition. If $\varphi x=z'$, then $B_x=
\{x\}$ and $\varphi(B_x)=\varphi(\{x\})=\{\varphi x\}=B'_x\,$. If
$\varphi x\ne z'$, then $\varphi(B_x)=\varphi(B(x,fx))\subseteq
B(\varphi x,z') =B'_x$ holds by (UAT2). Hence, (AT1) is satisfied.

In order to prove that (AT2) is satisfied, we assume that $\varphi x\ne
z'$. By (UAT1), we have that $d'(\varphi fx,z') <d'(\varphi x,z')$,
which by (\ref{cb}) implies that $B'_{fx}=B(\varphi fx,z')\subsetuneq
B(\varphi x,z') =B'_x\,$. Thus, (AT2) holds. We also see that $\varphi
x\notin B(\varphi fx,z')$.

Now we show that $f$ is strongly contracting on orbits. (SC1) holds by
definition, and (SC2) holds trivially when $x=fx$. So we assume that
$x\ne fx$. We run the above construction again with $fx$ in place of $x$
to obtain $f^2 x$. If $\varphi fx=z'$, then $f^2 x=fx$ and $B_{fx}=
\{fx\} \subsetuneq B_x$ since $x\ne fx$ by assumption. Now assume that
$\varphi fx\ne z'$. Then by what we have already shown, $\varphi(B_{fx})
\subseteq B'_{fx}$, so $x\notin B_{fx}\,$. From (UAT3), where we replace
$t,x,y$ by $x,fx,f^2x$, we infer that $d(x,fx)$ and $d(fx,f^2x)$ are
comparable. Therefore, $B_x\subseteq B_{fx}$ or $B_{fx}\subseteq B_x\,$.
Since $x\notin B_{fx}\,$, it follows that $B_{fx}\subsetuneq B_x\,$,
which proves that $f$ is strongly contracting on orbits.

We have shown that $z'$ is a weak $f$-attractor for $\varphi$. Our
theorem will thus follow from Theorem~\ref{AT1} once we have proved that
$f$ also satisfies (SC3). Take an $f$-nest $\cN$ and some $z\in
\bigcap\cN$. We have to show that $B_z\subseteq\bigcap\cN$, that is,
$B_z\subseteq B_x$ for all $B_x\in\cN$. If $\varphi x=z'$, then $z\in
B_x=\{x\}$, whence $z=x$ and $B_z=B_x\,$. Hence we will again assume
that $\varphi x\ne z'$.

Since $z\in\bigcap\cN\subseteq B_{fx}$, we have that
\[
\varphi z\,\in\,\varphi(B_{fx})\>\subseteq\>
B'_{fx}\>=\>B(\varphi fx,z')\;.
\]
It follows that
\[
\varphi(B_z)\>\subseteq\>B'_z\>=\>B(\varphi z,z')\>\subseteq\>
B(\varphi fx,z')\;.
\]
But we have already shown that $\varphi x \notin
B(\varphi fx,z')$, hence $x\notin B_z\,$. We have that
\begin{equation}                          \label{zfxx}
d'(\varphi z,z')\>\leq\>d'(\varphi fx,z')\><\>d'(\varphi x,z')
\end{equation}
and
\begin{equation}                            \label{B'zfx}
\varphi(B(x,z))\>\subseteq\>\varphi(B_x)\>\subseteq\> B'_x\>=\>
B(\varphi x,z')\;.
\end{equation}
From (UAT3), where we replace $t,x,y$ by $x,z,fz$, we infer that
$d(x,z)$ and $d(z,fz)$ are comparable. We find that $B(x,z)\subseteq
B(z,fz)$ or $B(z,fz)\subseteq B(x,z)\,$.
Since $x\notin B_z\,$, we obtain that $B_z\subseteq B(x,z)\subseteq
B_x\,$.
\end{proof}

Note that our condition (UAT3) is somewhat stronger than condition (8)
of [PR] because we do not start with a given function $f$ (which is
called $g$ in [PR]), but construct it in our proof. Rewritten in our present
notation, condition (8) of [PR] states:
\sn
{\bf (UAT3$'$)} \ if $d'(\varphi z,z')<d'(\varphi x,z')$ and
$B(z,fz)\cap B(x,fx)\ne\emptyset$, then $d(z,fz)<d(x,fx)$.
\sn
If the function $f$, strongly contracting on orbits, is already given,
then condition (UAT3) in Theorem~\ref{MTattr} can be replaced by
(UAT3$'$). Let us show how (UAT3$'$) is used at the end of the proof of
Theorem~\ref{MTattr} to deduce that (SC3) holds. For $z\in \bigcap\cN$,
we have to show that $B_z\subseteq B_x$ for all $B_x\in\cN$. As in the
above proof, one shows that (\ref{zfxx}) holds. Since $z\in
\bigcap\cN\subseteq B_x=B(x,fx)$, we have that $B(z,fz)\cap B(x,fx)
\ne\emptyset$. Hence by (UAT3$'$), $d(z,fz)< d(x,fx)$, which by
(\ref{umball}) implies that $B_z\subseteq B_x\,$. Observe that the
condition ``$d(z,fz)<d(x,fx)$'' in (UAT3$'$) can be replaced by the
weaker condition that $d(z,fz)$ and $d(x,fx)$ are comparable.

%
%
\section{Completeness by stages}            \label{sectcbs}
In this section, we consider valued fields $(K,v)$. In order to be
compatible with the way we have presented ultrametrics in the previous
sections, we write the valuation $v$ multiplicatively, that is, the
value group is a multiplicatively written ordered abelian group with
neutral element $1=v(1)$, and we add a smallest element $0=v(0)$. The
axioms for a valuation in this notation are:
\sn
{\bf (VF1)} \ $v(x)=0\>\Leftrightarrow x=0$,\n
{\bf (VF2)} \ $v(xy)\>=\>v(x)v(y)$,\n
{\bf (VF3)} \ $v(x+y)\>\leq\max\{v(x),v(y)\}$.
\sn
The underlying ultrametric is obtained by setting $d(x,y):=v(x-y)$.

\pars
We will work in the valuation ideal ${\cal M}=\{x\in K\mid v(x)<1\}$ of
$(K,v)$ since this facilitates the notation, and the typical
applications of the fixed point theorem we are going to prove can be
made to deal with functions $f:{\cal M}\rightarrow {\cal M}$.

We assume the reader to be familiar with the theory of pseudo Cauchy
sequences (see for instance \cite{[KA]}). P.~Ribenboim
introduced in \cite{[R]} the notion of \bfind{distinguished pseudo
Cauchy sequence}. For a pseudo Cauchy sequence $(a_{\nu})_{\nu<\lambda}$
in ${\cal M}$, indexed by a limit ordinal $\lambda$, the original
definition of ``distinguished'' is equivalent to the following: for
every $\mu<\lambda$ there is $\nu<\lambda$ such that
\[
v(a_{\nu+1}-a_\nu) \leq v(a_{\mu+1}-a_\mu)^2\;.
\]
The valued field $(K,v)$ is called \bfind{complete by stages} if every
distinguished pseudo Cauchy sequence in ${\cal M}$ has a pseudo limit in
$K$. Ribenboim proves in \cite{[R]} that every such field is henselian. For
this proof, one can use a theorem like the following:

\begin{theorem}                             \label{FPTcbs}
Take a valued field $(K,v)$ that is complete by stages and a contracting
function $f:{\cal M}\rightarrow {\cal M}$. If for every $x\in K$ there is
$j\in \N$ such that
\[
v(f^j x-f^{j+1} x)\>\leq\>v(x-fx)^2\;,
\]
then $f$ has a fixed point in $K$.
\end{theorem}
\sn
Note that $v(x-fx)<1$ since $x,fx\in {\cal M}$. In the theorem,
${\cal M}$ can be replaced by any ultrametric ball $B$ in $K$
as long as $v(x-fx)<1$ for all $x\in B$.

In order to deduce this theorem from Theorem~\ref{GFPT2}, we need to
show the connection between completeness by stages and nests of balls
with certain properties. We will call a nest $\cN$ of ultrametric balls
in ${\cal M}$ \bfind{distinguished} if for all $x,y\in {\cal M}$ with
$B(x,y)\in\cN$
there are $x',y'\in {\cal M}$ such that $v(x'-y')\leq v(x-y)^2$ and
$B(x',y')\in\cN$. Then the following holds:

\begin{lemma}                               \label{cbs=dni}
A valued field $(K,v)$ with valuation ideal ${\cal M}$ is complete by
stages if and only if each distinguished nest of ultrametric balls in
${\cal M}$ has a nonempty intersection.
\end{lemma}
\sn
In this lemma, ${\cal M}$ can be replaced by the valuation ring
${\cal O}$, and also by the ultrametric balls $a+{\cal M}$ and $a+
{\cal O}$ for all $a\in K$. The proof of the lemma is similar to the
proof of the fact that an ultrametric space is spherically complete if
and only if every pseudo Cauchy sequence in this space has a pseudo
limit. It is based on the following easy observation:

\begin{lemma}                               \label{wosn}
Every nest of balls admits, in the ordering
given by reverse inclusion, a cofinal well ordered subnest.
\end{lemma}
Now the method of proof is to associate such a coinitial well ordered
subnest with a pseudo Cauchy sequence such that each pseudo limit will
lie in the intersection of the nest --- and vice versa. Since every
subnest of a distinguished nest is again distinguished, this will give
rise to a distinguished pseudo Cauchy sequence. Conversely, if
$(a_{\nu})_{\nu<\lambda}$ is a distinguished pseudo Cauchy sequence in
${\cal M}$, then $\{B(a_\nu,a_{\nu+1})\mid \nu<\lambda\}$ is a
distinguished nest of balls in ${\cal M}$.

\sn
{\bf Proof of Theorem~\ref{FPTcbs}:} \
As before, we set $B_x:=B(x,fx)$ and note that $x\in B_x\,$. Since $f$
is contracting, we have that $v(fx-f^2x)\leq v(x-fx)$, which by
(\ref{umball}) implies that $B_{fx}\subseteq B_x\,$. If $fx\ne x$, then
by assumption, there is $i\in\N$ such that $v(f^i x-f^{i+1} x)\>\leq\>
v(x-fx)^2<v(x-fx)$; this yields that $x\notin B_{f^ix}$ and $B_{f^ix}
\subsetuneq B_x\,$. We have now proved that $f$ is strongly contracting
on orbits. Since $f$ is contracting, this implies by way of
Lemma~\ref{csco} that $f$ also satisfies (SC3).

Take an $f$-nest $\cN$. Every ball in $\cN$ is of the form $B(x,fx)$ and
there is some $i\in\N$ such that $v(f^i x-f^{i+1} x)\>\leq\> v(x-fx)^2$.
By the definition of an $f$-nest, $B_{f^ix}\in\cN$. Thus, $\cN$ is
distinguished. Since $(K,v)$ is assumed to be complete by stages,
Lemma~\ref{cbs=dni} shows that $\cN$ has a nonempty intersection. Since
$f$ also satisfies (SC3) and is strongly contracting on orbits, this
implies that the conditions of Theorem~\ref{GFPT2} are satisfied, and we
obtain the assertion of Theorem~\ref{FPTcbs}.

%
%
\section{Banach's Fixed Point Theorem}      \label{sectBFT}
Banach's Fixed Point Theorem states that every strictly contracting
function on a complete metric space $(X,d)$ has a unique fixed point. A
function $f:X\rightarrow X$ is called \bfind{strictly contracting} if
there is a positive real number $C<1$ such that $d(fx,fy)\leq Cd(x,y)$
for all $x,y\in X$. We will show now how Banach's Fixed Point Theorem
fits into the setting of Theorems~\ref{GFPT2} and~\ref{GFPT1}. We work
in the ball space $(X,\cB)$ where $\cB$ consists of all balls $\{y\in
X\mid d(x,y)\leq r\}$ for $x\in X$ and $r\in\R^{\geq 0}$. This ball
space is spherically complete since $(X,d)$ is complete.

We will prove the existence of fixed points under the slightly more
general assumption that $f$ is
\sn
1) \ \bfind{contracting}, that is, $d(fx,fy)\leq d(x,y)$ for all
$x,y\in X$, and
\n
2) \ \bfind{strictly contracting on orbits}, that is, there is a
positive real number $C<1$ such that $d(fx,f^2x)\leq Cd(x,fx)$ for all
$x\in X$.

\pars
Take any $x\in X$. Then
\begin{eqnarray*}
d(x,f^ix) & \leq & d(x,fx)+d(fx,f^2x)+\ldots+d(f^{i-1}x,f^ix)\\
 & \leq & d(x,fx)(1+C+C^2+\ldots+C^{i-1})\\
 & \leq & d(x,fx)\sum_{i=0}^{\infty}C^i\>=\>\frac{d(x,fx)}{1-C}\;.
\end{eqnarray*}
Hence if we set
\[
B_x\>:=\>\left\{y\in X\mid d(x,y)\,\leq\,\frac{d(x,fx)}{1-C}\right\}\;,
\]
then $f^ix\in B_x$ for $i\geq 0$. In particular, $x\in B_x\,$, hence
(SC1) holds.

We wish to show that $B_{fx}\subseteq B_x\,$. Take any
$y\in B_{fx}\,$. Then
\begin{eqnarray*}
d(x,y) & \leq & d(x,fx)\,+\,d(fx,y)\>\leq\>d(x,fx)\,+\,
\frac{d(fx,f^2x)}{1-C}\\
& \leq & d(x,fx)\,+\, \frac{C}{1-C}\,d(x,fx)\>=\>\frac{d(x,fx)}{1-C}\;.
\end{eqnarray*}
Thus, $y\in B_x\,$, which proves our assertion.
\pars
Since $C<1$, there is some $i\geq 1$ such that
\[
\frac{C^i}{1-C}\><\>\frac{1}{2}\;.
\]
Then
\[
\frac{d(f^ix,f^{i+1}x)}{1-C}\>\leq\>\frac{C^i}{1-C}\,d(x,fx)\><\>
\frac{1}{2}\,d(x,fx)\;,
\]
which implies that $x$ and $fx$ cannot both lie in $B_{f^ix}\,$.
Therefore, $B_{f^ix}\subsetuneq B_x$ and we have now proved that (SC2)
holds.

Next, we show that also (SC3) holds. Take an $f$-nest $\cN$ and assume
that $z\in \bigcap\cN$. Pick any $B_x\in\cN$ and $i>0$. Since $f$ is
contracting, $d(f^i x,fz)\leq d(f^{i-1}x,z)=d(z,f^{i-1}x)$. Using that
$z\in\cN \subseteq B_{f^ix}$ for all $i$, we compute:
\begin{eqnarray*}
d(z,fz) & \leq & d(z,f^ix)+d(f^ix,fz)\>\leq\>d(z,f^ix)+d(z,f^{i-1}x)\\
& \leq & \frac{d(f^ix,f^{i+1}x)}{1-C}+
\frac{d(f^{i-1}x,f^i x)}{1-C} \\
& \leq & \frac{C^i}{1-C}\,d(x,fx)+\frac{C^{i-1}}{1-C}\,d(x,fx)\>=\>
C^{i-1}\frac{C+1}{1-C}\,d(x,fx)\;.
\end{eqnarray*}
Since $\lim_{i\rightarrow\infty}C^i=0$, we obtain that $fz=z$, so we
have found a fixed point. It follows that $B_z\subseteq \bigcap\cN$ (in
fact, $\bigcap\cN=\{z\}=B_z$). This shows that $f$ is self-contractive.
%

\pars
Note that if $f$ is strictly contracting, then the fixed point is
unique. Indeed, if there were distinct fixed points $x,y$, then
$d(x,y)=d(fx,fy)<d(x,y)$, a contradiction.

%
%
\section{The case of ordered abelian groups and fields}
\label{sectna}
In this section we will discuss various forms of fixed point theorems in
the case of ordered abelian groups and fields. Here, we always mean that
the ordering is total. The ordering induces a natural valuation; we will
recall its definition in Section~\ref{sectprel}. This valuation is
nontrivial if and only if the ordering is nonarchimedean. Since the
valuation induces an ultrametric, our ultrametric fixed point theorems
can be translated to the present case. We will do this in
Section~\ref{sectordu}.

However, the most natural idea to derive a ball space from the ordering
of an ordered abelian group $(G,<)$ is to define the \bfind{order balls}
in $G$ to be the sets of the form
\[
B_o(g;r)\>:=\>\{z\in G\mid |g-z|\leq r\}
\]
for arbitrary $g\in G$ and nonnegative $r\in G$. We set
\[
\cB_o\>=\>\cB_o(G,<)\>:=\> \{B_o(g;r) \mid g\in G,\,0\leq r\in G\}\;.
\]
Then $(G,\cB_o)$ is the \bfind{order ball space} associated with
$(G,<)$. In Section~\ref{sectob} we will state a fixed point theorem
corresponding to the order ball space.

\pars
Before we continue, we need some preliminaries and general background.

%
%
\subsection{Preliminaries on non-archimedean
ordered abelian groups and fields}          \label{sectprel}
Take an ordered abelian group $(G,<)$. Two elements $a,b\in G$ are
called \bfind{archimedean equivalent} if there is some $n\in\N$ such
that $n|a|\geq |b|$ and $n|b|\geq |a|$. The ordered group $(G,<)$ is
archimedean ordered if all nonzero elements are archimedean equivalent.
If $0\leq a<b$ and $na<b$ for all $n\in\N$, then ``$a$ is
infinitesimally smaller than $b$'' and we will write $a\ll b$.

We define the \bfind{natural valuation} of $(G,<)$ as follows. We denote
by $va$ the archimedean equivalence class of $a$. The set of archimedean
equivalence classes is ordered as follows: $va<vb$ if and only if
$|a|<|b|$ and $a$ and $b$ are not archimedean equivalent, that is,
if $n|a|<|b|$ for all $n\in\N$. We write $0:=v0\,$; this is the minimal
element in the totally ordered set of equivalence classes. The function
$a\mapsto va$ is a group valuation on $G$, i.e., it satisfies
$vx=0\Leftrightarrow x=0$ and the ultrametric triangle law
\sn
{\bf (UT)} \ $v(x-y)\>\leq\max\{vx,vy\}$.
\sn
The natural valuation induces an ultrametric defined by
%
\[
d(x,y)\>:=\>v(x-y)
\]
and hence an ultrametric ball space, with the set $\cB_u=\cB_u(G,<)$ of
balls $B_u(x,y)$ defined as in Section~\ref{sectult}. We will call
$(G,\cB_u)$ the \bfind{(natural) ultrametric ball space} of $(G,<)$.
Note that all ultrametric balls are cosets of convex subgroups in $G$.

\pars
If $(K,<)$ is an ordered field, then we consider the natural
valuation on its ordered additive group and define $va\cdot vb:=v(ab)$.
This turns the set of archimedean classes into a multiplicatively
written ordered abelian group, with neutral element $1:=v1$ and inverses
$(va)^{-1}=v(a^{-1})\,$. In this way, $v$ becomes a field valuation
(with multiplicatively written value group). It is the finest valuation
on $K$ which is compatible with the ordering. The residue field $Kv:=
{\cal O}/{\cal M}$ is archimedean ordered, hence by the version of the
Theorem of H\"older for ordered fields, it can be embedded in the
ordered field $\R$. Via this embedding, we will always identify it with
a subfield of $\R$.

We know from \cite[Theorem 6]{[KA]} that $K$ can be embedded in the
power series field with exponents in the value group and coefficients in
the residue field of its natural valuation. (A nontrivial factor set may
be needed, unless the positive part of the residue field is closed under
radicals, which for instance is the case if $K$ is real closed.)
Moreover, the ultrametric ball space of $K$ is spherically complete if
and only if the embedding is onto.

%
%
\subsection{Ultrametric balls}              \label{sectordu}
Via the natural valuation, the ultrametric fixed point theorems provide
fixed point theorems for ordered abelian groups and fields. A valuation
is called \bfind{spherically complete} if its associated ultrametric
ball space is spherically complete, and it is called \bfind{complete by
stages} if its associated ultrametric ball space is complete by stages.

Take an ordered abelian group $G$ and a function $f:G\rightarrow G$. It
will be called \bfind{o-contracting} if
\[
|fx-f^2x|\leq |x-fx|
\]
for all $x\in G$; note that an o-contracting function is also
contracting in the ultrametric sense. The property
\begin{equation}                            \label{oscoo}
x\ne fx \>\Longrightarrow\>\exists i\geq 1:\>
|f^ix-f^{i+1}x|\>\ll\> |x-fx|
\end{equation}
implies the property (\ref{scoo}).
%
%
Hence, the following theorem is an immediate consequence of
Theorem~\ref{UFPT}. In order to obtain it for an ordered field $K$, take
$G$ to be the additive group of $K$.

\begin{theorem}
Take an ordered abelian group $G$ whose natural valuation is spherically
complete. If $f:G\rightarrow G$ is an o-contracting function that
satisfies (\ref{oscoo}), then it has a fixed point.
\end{theorem}

Take an ordered field $K$ and let $v$ denote its natural valuation. Then
the valuation ideal ${\cal M}$ of $v$ is the set of all infinitesimals
of $K$, that is, the elements $x\in K$ such that $|x|\ll 1$. The next
theorem follows directly from Theorem~\ref{FPTcbs}. We suspect
that the theorem becomes false if ${\cal M}$ is replaced by the
valuation ring ${\cal O}$.

\begin{theorem}
Take an ordered field $K$ with ${\cal M}$ the set of its
infinitesimals, whose natural valuation is complete by stages.
If $f:{\cal M}\rightarrow {\cal M}$ is an o-contracting function and for
every $x\in {\cal M}$ there is $j\in \N$ such that
\[
|f^j x-f^{j+1} x|\>\leq\>|x-fx|^2\;,
\]
then $f$ has a fixed point in ${\cal M}$.
\end{theorem}

%
%
\subsection{Order balls}                    \label{sectob}
Take an ordered abelian group $(G,<)$. We call a function $f:G
\rightarrow G$ \bfind{strictly o-contracting on orbits} if there is a
positive rational number $\frac{m}{n}<1$ with $m,n\in\N$ such that
$n|fx-f^2x| \leq m|x-fx|$ for all $x\in X$. Note that $n|fx-f^2x|\leq
m|x-fx|$ holds if and only if $|fx-f^2x|\leq \frac{m}{n}|x-fx|$ in the
divisible hull of $G$ with the unique extension of the ordering.

The following theorem is a consequence of Theorem~\ref{GFPT1}:
\begin{theorem}                             \label{OAG}
Suppose that the order ball space $(G,\cB_o)$ associated with the
ordered abelian group $(G,<)$ is spherically complete. Then every
o-contracting function on $G$ which is strictly o-contracting on
orbits has a fixed point.
\end{theorem}
\begin{proof}
We compute in the divisible hull of $G$. We choose $C\in\Q$ such that
$\frac{m}{n}<C<1$. With $d(x,y):=|x-y|$, the computations of
Section~\ref{sectBFT} remain valid, and we may define the balls $B_x$
in exactly the same way. They will then again satisfy (SC1) and (SC2).
However, we have to work a little bit harder to show that (SC3) holds.
In the present case, we cannot conclude that $d(z,fz)=0$.
But for given $B_x\in\cN$ we can conclude that
\[
d(z,fz)\>\leq\> C^{i-1}\frac{C+1}{1-C}\,d(x,fx)
\]
for all $i>0$ (which means that $d(z,fz)\ll d(x,fx)$). We
choose $i$ so large that
\[
C^{i-1}\frac{C+1}{1-C} \>\leq\> C-\frac{m}{n}\;.
\]
Then for every $y\in B_z\,$,
\[
d(z,y)\>\leq\>\frac{d(z,fz)}{1-C}\>\leq\>C^{i-1}\frac{C+1}{1-C}\,d(x,fx)
\>\leq\> \left(C-\frac{m}{n}\right)d(x,fx)\;.
\]
Using that $z\in\bigcap\cN\subseteq B_{f^2x}\,$, we compute:
\begin{eqnarray*}
d(x,y) & \leq & d(x,fx)\,+\,d(fx,f^2x)\,+\,d(f^2x,z)\,+\,d(z,y)\\
 & \leq & \left(1+\frac{m}{n}\right)d(x,fx)\,+\,\frac{C^2}{1-C}\,d(x,fx)
\,+\,\left(C-\frac{m}{n}\right)d(x,fx) \\
 & = & (1+C)d(x,fx)\,+\,\frac{C^2}{1-C}\,d(x,fx)
\>=\>\frac{d(x,fx)}{1-C}\;.
\end{eqnarray*}
Hence $y\in B_x\,$. We have proved that $B_z\subseteq B_x\,$, as
required. Altogether, we have shown that $f$ is self-contractive. Now
the assertion of our theorem follows from Theorem~\ref{GFPT1}.
\end{proof}

We leave it to the reader to formulate a corresponding version of
Theorem~\ref{GFPT2}.

\pars
In the case of an ordered field $(K,<)$, we may actually give a slightly
more general definition of ``strictly contracting on orbits'' by
requiring that there is an element $C\in K$ such that $0<C<1$, $v(1-C)
=1$, and $|fx-f^2x|\leq C|x-fx|$ for all $x\in X$, as this condition in
fact implies the condition of the original definition; indeed, the
condition ``$v(1-C)=1$'' means that $C$ is not infinitesimally close to
$1$ and therefore, there is some $C'\in\Q$ such that $C\leq C'<1$.
Working in the ordered additive group of the field, one then immediately
obtains from Theorem~\ref{OAG} the corresponding version for ordered
fields.

At first glance, spherical completeness of the order ball space appears
to be a very strong condition. We note:

\begin{lemma}                               \label{asco}
An archimedean ordered abelian group has a spherically complete
order ball space if and only if it is isomorphic to the additive group
of the integers or of the reals with the canonical ordering.
\end{lemma}
\begin{proof}
First, observe that every nest $\cN$ of order balls in the integers
contains a smallest ball, so the integers have a spherically complete
order ball space.

Next, take a nest $\cN$ of order balls in the reals. Pick a ball $B$ in
this nest and consider the nest $\cN_0$ of all balls in $\cN$ that are
contained in $B$; this nest has the same intersection as $\cN$. The
order ball $B$ is compact in the order topology of $\R$. Since all balls
in $\cN_0$ are closed in the order topology, Lemma~\ref{cHsc} shows that
the intersection of $\cN_0$ is nonempty.

From what we have proved it follows that every ordered abelian group
which is isomorphic to the additive ordered group of the integers or the
reals has a spherically complete order ball space.

For the converse, we use that by the Theorem of H\"older, every
archimedean ordered abelian group $G$ can be embedded in the additive
ordered group of the reals. We identify $G$ with its image under this
embedding and show that if $G$ is spherically complete w.r.t.\ the order
balls, but not isomorphic to $\Z$, then $G=\R$. Since $G$ is not
isomorphic to $\Z$, it is dense in $\R$. Hence for every $a\in\R$ there
is an increasing sequence $(g_i)_{i\in\N}$ in $G$ converging to $a$. By
density, we can also find a decreasing sequence $(r_i)_{i\in\N}$ in $G$
converging to $0$ such that $(g_i+2r_i)_{i\in\N}$ is a decreasing
sequence in $G$ converging to $a$. Then $\{B_o(g_i+r_i\,;r_i)\mid
i\in\N\}$ is a nest of order balls in $G$. Since $G$ has a spherically
complete order ball space, this nest has a nonempty intersection. As
this intersection can only contain the element $a$, we find that $a\in G$.
\end{proof}

One might think that spherical completeness w.r.t.\ the order balls
implies cut completeness, in which case $\Z$ and $\R$ would be the only
ordered groups with this property. But fortunately, this is not the
case. In \cite{[S]}, Saharon Shelah has shown that every ordered field
is contained in one that has a spherically complete order ball space
(see also \cite{[KKS]} for a power series field construction of such
fields). So there are arbitrarily large ordered fields (and hence also
ordered abelian groups) in which our above fixed point theorem holds.

We will study the structure of such ordered fields and abelian groups
more closely in a subsequent paper. But we include some basic facts
here. The next lemma shows that spherical completeness of the order
ball space is stronger than spherical completeness of the
ultrametric ball space.

\begin{proposition}                               \label{scoscu}
If an ordered abelian group $(G,<)$ has a spherically complete order
ball space, then it has a spherically complete ultrametric ball space.
\end{proposition}
\begin{proof}
Assume that the ordered abelian group $(G,<)$ has a spherically complete
order ball space and take a nest $\cN$ of ultrametric balls in $G$. If
this nest contains a smallest ball, then its intersection is nonempty
and there is nothing to show. So we assume that $\cN$ does not contain a
smallest ball.

In view of Lemma~\ref{wosn}, we may assume that the nest is of the form
$\{B_u(x_\mu,y_\mu)\mid \mu<\kappa\}$, where $\kappa$ is a regular
cardinal (i.e., equal to its own cofinality), and $B_u(x_\nu,y_\nu)
\subsetuneq B_u(x_\mu,y_\mu)$ whenever $\mu<\nu<\kappa$. The latter
implies that $v(x_\nu-y_\nu)<v(x_\mu-y_\mu)$. For each $\mu<\kappa$, we
define an order ball
\[
B^\mu\>:=\>B_o(x_{\mu+1}\,;|x_\mu-y_\mu|)\>\subseteq B_u(x_\mu,y_\mu)\;.
\]
Since $v(x_{\mu+1}-y_{\mu+1})<v(x_\mu-y_\mu)$ implies that $|x_{\mu+1}
-y_{\mu+1}| \ll |x_\mu-y_\mu|$, we find that $B_u(x_{\mu+1}, y_{\mu+1})
\subseteq B^\mu$. It follows that $\{B^\mu\mid \mu<\kappa\}$ is a nest
of balls and its intersection is equal to the intersection of $\cN$.
Since $G$ has a spherically complete order ball space by assumption,
this intersection is nonempty.
\end{proof}

Every ordered field contains the field $\Q$ of rational numbers as a
subfield. We use this fact in the following lemma:

\begin{lemma}                               \label{qobrfR}
An ordered field is spherically complete w.r.t.\ the order balls
\begin{equation}                            \label{qordballs}
B_o(q\,;r)\,,\quad q,r\in\Q \;\mbox{ with }\; r>0\;,
\end{equation}
if and only if its residue field under the natural valuation is $\R$.
\end{lemma}
\begin{proof}
Assume that the ordered field $(K,<)$ is spherically complete w.r.t.\
the order balls of the form (\ref{qordballs}). We pick some $a\in \R$
and wish to show that $a\in Kv$. Using the fact that $\Q$ is dense in
$\R$, we take a nest $\{B_o(g_i+r_i\,;r_i) \mid i\in\N\}$ of order balls
in $G=\Q\subseteq Kv$ as in the proof of Lemma~\ref{asco}. Taken as a
nest in $Kv$, its intersection is either empty or contains only $a$.

The residue map from ${\cal O}$ to $Kv$ induces an isomorphism from the
subfield $\Q\subseteq {\cal O}$ of $K$ onto the subfield $\Q$ of $\R$.
Via this isomorphism, we can see the elements $g_i$ and $r_i$ as
elements of $K$, and taken in $K$, $\{B_o(g_i\,;r_i) \mid i\in\N\}$ is a
nest of order balls of the form (\ref{qordballs}). By assumption, this
nest has a nonempty intersection; take $b\in K$ to lie in this
intersection. Then its residue lies in the intersection of the nest
$\{B_o(g_i\,;r_i) \mid i\in\N\}$ in $Kv$ and therefore must be equal to
$a$. We have shown that $a\in Kv$, and as it was an arbitrary element of
$\R$, we find that $Kv=\R$.

\pars
For the converse, assume that the ordered field $(K,<)$ has residue
field $\R$ under its natural valuation. Take a nest $\{B_o(q_i\,;r_i)
\mid i\in I\}$ in $K$ of balls of the form (\ref{qordballs}). Via the
residue map, we can view $q_i,r_i$ as elements of $Kv$, and we can view
$\{B_o(q_i\,;r_i) \mid i\in I\}$ as a nest of order balls in $Kv=\R$. By
Lemma~\ref{asco}, it has a nonempty intersection. Since the residue
map is surjective, there is an element $a\in K$ whose residue lies in
this intersection. We can take $a$ in the ball $B_o(q_i\,;r_i)$ in $K$,
for some $i\in I$. Suppose that $a$ does not lie in the intersection of
the nest in $K$. Then there is a smaller ball $B_o(q_j\,;r_j)$ in which
$a$ does not lie. Since the residue of $a$ lies in the intersection of
the nest in $Kv$, there must be some $a'$ in the ball $B_o(q_j\,;r_j)$
in $K$ which has the same residue as $a$. We assume without loss of
generality that $a>q_j+r_j\,$; the case of $a<q_j-r_j$ is symmetrical.
Then $a'<q_j+r_j<a$, and as $a$ and $a'$ have the same residue, this
must also be the residue of $b:=q_j+r_j$. Now we show that $b$ lies in
the intersection of the nest in $K$. If this were not true, then there
would exist a ball $B_o(q_k\,;r_k)$ which does not contain $b$, so
$b>q_k+r_k$ or $b<q_k-r_k$. Since $b,q_k,r_k\in\Q,$ it would follow that
no element in $B_o(q_k\,;r_k)$ could have the same residue as $b$ and
$a$, which leads to a contradiction. So we find that also the nest in
$K$ has a nonempty intersection. This proves that $(K,<)$ is spherically
complete w.r.t.\ the order balls of the form (\ref{qordballs}).
\end{proof}

From Proposition~\ref{scoscu} and Lemma~\ref{qobrfR}, we obtain the
following theorem.
\begin{theorem}
Every ordered field with a spherically complete order ball space is
isomorphic to a power series field with residue field $\R$.
\end{theorem}

The converse is not true. For example, the power series field with value
group $\Q$ and residue field $\R$ does not have a spherically complete
order ball space. We will give a full characterization of the
ordered fields with spherically complete order ball spaces in a
subsequent paper, making use of additional conditions on the value
group.

%
%
\subsection{Hybrid ball spaces}          \label{secthyb}
Using the flexibility of our notion of ball space, we will now give a
simple characterization of the ordered fields that are power series
fields with residue field $\R$. We just have to enlarge the underlying
ultrametric ball space of an ordered field by a suitable set of order
balls. We will make use of two very easy principles:

\begin{lemma}                               \label{sb,ub}
1) \ If the ball space $(X,\cB)$ is spherically complete and
$\cB'\subseteq \cB$, then also $(X,\cB')$ is spherically complete.
\sn
2) \ If $\cB_i\,$, $1\leq i\leq n$, are collections of subsets of $X$,
then the ball space $(X,\bigcup_{i=1}^{n}\cB_i)$ is spherically complete
if and only if all of the ball spaces $(X,\cB_i)$, $1\leq i\leq n$, are
spherically complete.
\end{lemma}
\sn
We leave the easy proofs to the reader. Just note that for every nest of
balls in $(X,\bigcup_{i=1}^{n}\cB_i)$ there must be an $i$ and a cofinal
(under reverse inclusion) subnest of balls that all lie in $\cB_i\,$.

\pars
We take an ordered field $(K,<)$ and define the \bfind{hybrid ball space
of $(K,<)$} to be $K$ together with the union of the set of ultrametric
balls with the set of order balls of the form (\ref{qordballs}), that
is,
\[
\cB_h(K,<)\>:=\> \cB_u(G,<)\,\cup\,\{B_o(q;r)\mid q,r\in\Q \mbox{ with }
r>0\}\;.
\]
From Proposition~\ref{scoscu} and Lemma~\ref{sb,ub} it follows that if
the order ball space of an ordered field is spherically complete, then
so is its hybrid ball space.

From Lemmas~\ref{qobrfR} and~\ref{sb,ub}, we obtain the following
theorem:

\begin{theorem}                             \label{charpsfR}
An ordered field $(K,<)$ is isomorphic to a power series field with
residue field $\R$ if and only if its hybrid ball space $(K,\cB_h(K,<))$
is spherically complete.
\end{theorem}

A different characterization of power series fields with residue field
$\R$ was obtained by Ron Brown (cf.\ \cite[Theorem 1.7]{b}), using the
notion of ``ultracomplete w.r.t.\ an extended absolute value''. This
notion appears to be closely related to spherical completeness of
suitable ball spaces; details will be worked out in a subsequent paper.

\newcommand{\lit}[1]{\bibitem{#1}}


\begin{thebibliography}{99}
\bibitem{ap} General topology I. Basic concepts and constructions.
Dimension theory. Translation from the Russian edited by A. V.
Arkhangel'skii and L. S. Pontryagin. Encyclopaedia of Mathematical
Sciences, {\bf 17}. Springer-Verlag, Berlin, 1990
\bibitem{b} Brown, R.$\,$: {\it Extended prime spots and quadratic forms},
Pacific J. Math. {\bf 51} (1974), 379--395
\bibitem{tt} Carter, A.\ -- Lithio, D.\ -- Niichel, R.\ -- Tager, T.$\,$:
{\it Beta Spaces: New Generalizations of Typically-Metric Properties},
preprint, arXiv:1011.4027v1 [math.GN]
\bibitem{[KA]} {Kaplansky, I.$\,$: {\it Maximal fields with valuations I},
Duke Math.\ J.\ {\bf 9} (1942), 303--321}
\bibitem{[KU3]} {Kuhlmann, F.-V.$\,$: {\it Maps on ultrametric spaces,
Hensel's Lemma, and differential equations over valued fields}, Comm. in
Alg. {\bf 39} (2011), 1730--1776}
\bibitem{[KU4]} {Kuhlmann, F.-V.$\,$: {\it Valuation theory}, book in
preparation. Preliminary versions of several chapters available at
{\tt http://math.usask.ca/$\,\tilde{ }\,$fvk/Fvkbook.htm}}
\bibitem{[KKS]} {Kuhlmann, F.-V.\ --- Kuhlmann, K.\ --- Shelah, S.$\,$:
{\it Symmetrically complete ordered sets, abelian groups and fields}, in
preparation}
\bibitem{[P]} {Prie{\ss}-Crampe, S.$\,$: {\it Der Banachsche Fixpunktsatz
f\"ur ultrametrische R\"aume}, Results in Mathematics {\bf 18} (1990),
178--186}
\bibitem{[PR1]} {Prie{\ss}-Crampe, S.\ -- Ribenboim, P.$\,$: {\it Fixed
Points, Combs and Generalized Power Series},
Abh.\ Math.\ Sem.\ Hamburg {\bf 63} (1993), 227--244}
\bibitem{[PR2]} {Prie{\ss}-Crampe, S.\ -- Ribenboim, P.$\,$: {\it Fixed
Point and Attractor Theorems for Ultrametric Spaces}, Forum Math.\ {\bf
12} (2000), 53--64}
\bibitem{[R]} Ribenboim, P.: Th\'eorie des valuations. Les
Presses de l'Uni\-versit\'e de Montr\'eal, Montr\'eal, 2nd ed. (1968)
\bibitem{[S]} Shelah, S.$\,$: {\it Quite Complete Real Closed Fields},
Israel J. Math. {\bf 142} (2004), 261--272
\bibitem{[SWJ]} Stepr\={a}ns, J.\ --   Watson, S. \ -- Just, W. $\,$:
{\it A topological Banach fixed point theorem for compact Hausdorff spaces},
Canad. Bull. Math.\ {\bf 37}(4)  (1994), 552--555

\end{thebibliography}
\end{document}